\documentclass{amsart}
\usepackage[all]{xy}
\usepackage{enumerate}
\usepackage{mathrsfs}
\usepackage{amssymb}
\usepackage{graphicx}
\usepackage{tikz-cd}

\theoremstyle{plain}
\newtheorem{theorem}{Theorem}[section]
\newtheorem{lemma}[theorem]{Lemma}

\newtheorem{proposition}[theorem]{Proposition}

\theoremstyle{definition}

\theoremstyle{remark}
\newtheorem{remark}[theorem]{Remark}

\newcommand\GL{\operatorname{GL}}

\newcommand\Hom{\operatorname{Hom}}

\newcommand\Sym{\operatorname{Sym}}

\newcommand\Ext{\operatorname{Ext}}

\newcommand\id{\operatorname{id}}
\newcommand\pr{\operatorname{pr}}

\newcommand\gr{\mathbf{gr}}

\newcommand\op{\mathrm{op}}
\newcommand\ab{\mathrm{ab}}
\newcommand\sgn{\mathrm{sgn}}

\newcommand\Z{\mathbb{Z}}
\newcommand\N{\mathbb{N}}
\newcommand\K{\Bbbk}

\newcommand\A{\mathbf{A}}

\newcommand\Lie{\mathcal{L}ie}

\newcommand\gpS{\mathfrak{S}}

\newcommand\Ob{\operatorname{Ob}}

\newcommand\calC{\mathcal{C}}

\newcommand\KMod{\K\textbf{-}\mathrm{Mod}}

\newcommand\catLie{\mathcal{C}at_{\mathcal{L}ie}}
\newcommand\catLieMod{\catLie\textbf{-}\mathrm{Mod}}
\newcommand\kgrop{\K\mathbf{gr}^{\operatorname{op}}}
\newcommand\kgropMod{\kgrop\textbf{-}\mathrm{Mod}}

\newcommand\catAss{\mathcal{C}at_{\mathcal{A}ss^{\mathrm{u}}}}

\newcommand\catLieC{\mathcal{C}at_{\mathcal{L}ie^{\mathrm{C}}}}

\newcommand\catLieCMod{\catLieC\textbf{-}\mathrm{Mod}}
\newcommand\AMod{\A\textbf{-}\mathrm{Mod}}
\newcommand\ub{\mathrm{uB}}

\newcommand{\adj}[4]{
\begin{tikzcd}[ampersand replacement=\&]
#1\arrow[r, shift left=1ex, "#3"{name=top}] \& #2\arrow[l, shift left=.5ex, "#4"{name=bottom}]
\arrow[phantom, from=top, to=bottom, , "\scriptscriptstyle\boldsymbol{\top}" rotate=180]
\end{tikzcd}
}

\title[Extensions between functors from Jacobi diagrams in handlebodies]{Extensions between functors from \\ Jacobi diagrams in handlebodies}
\author{Mai Katada}
\address{Graduate School of Mathematical Sciences, University of Tokyo, Tokyo 153-8914, Japan}
\email{mkatada@ms.u-tokyo.ac.jp}
\date{September 23, 2025}



\begin{document}

\begin{abstract}
The first Ext-groups between Schur functors in the category of modules over the $\K$-linearization $\kgrop$ of the opposite of the category of finitely generated free groups are computed for a filed $\K$ of characteristic $0$.
The $\K$-linear category $\A$ of Jacobi diagrams in handlebodies, which was introduced by Habiro and Massuyeau, has an $\N$-grading whose degree $0$ part identifies with the category $\kgrop$.
We compute the first Ext-groups in the category of $\A$-modules between simple $\A$-modules which are induced by Schur functors.
\end{abstract}
\maketitle

\section{Introduction}

Finitely generated free groups are important objects that appear in various fields of mathematics, such as topology, where they appear as fundamental groups.
There is a substantial body of literature on the functor category on the category $\gr$ of finitely generated free groups (or the opposite $\gr^{\op}$ of it) to the category of abelian groups (or the category of vector spaces) \cite{Djament--Vespa, Hartl--Pirashvili--Vespa, Powell--Vespa, Powellanalytic, Powellouter, PowellPassiMalcev, Kim--Vespa, Arone}.
Recently, the Ext-groups in the category of functors from $\gr$ to abelian groups have been studied \cite{Vespa, Arone, Kim--Vespa_Ext}, which are related to stable cohomology of automorphism groups of free groups with coefficients in reduced polynomial covariant functors \cite{Djament}.

We will focus on the functors to the category of vector spaces over a field $\K$ of characteristic $0$.
Powell \cite{Powellanalytic} proved that the category of analytic functors on $\gr^{\op}$ is equivalent to the category of modules over the $\K$-linear PROP $\catLie$ for Lie algebras.
Habiro and Massuyeau \cite{Habiro--Massuyeau} introduced the category $\A$ of Jacobi diagrams in handlebodies, which can be characterized as the $\K$-linear PROP for Casimir Hopf algebras.
The category $\A$ is $\N$-graded by the Casimir $2$-tensor, and the degree $0$ part of $\A$ identifies with the $\K$-linearization $\kgrop$ of $\gr^{\op}$.
Kim \cite{Kimtalk} has recently generalized the result of Powell to an equivalence of categories between the category of analytic functors on $\A$ and the category of modules over the $\K$-linear PROP $\catLieC$ for Casimir Lie algebras, which was introduced by Hinichi and Vaintrob \cite{Hinich--Vaintrob}.
Here, the category $\catLieC$ is $\N$-graded by the Casimir element, and the degree $0$ part of $\catLieC$ identifies with the category $\catLie$.
The author \cite{katadaAmod} independently obtained an adjunction between the category of analytic functors on $\A$ and the category of $\catLieC$-modules by using the category of extended Jacobi diagrams in handlebodies introduced in \cite{katada2}, which is equivalent to the equivalence of categories given by Kim.

The aim of this paper is to compute the first Ext-groups in the category $\AMod$ of $\A$-modules between simple $\A$-modules induced by Schur functors.
First, we compute the first Ext-groups in the category $\catLieCMod$ of $\catLieC$-modules. Via the equivalence of categories given by Kim, we obtain the first Ext-groups in $\AMod$. 
Then, we also give the direct computation of the first Ext-groups in the category $\AMod$. 

In the functor category on $\gr$, Vespa \cite{Vespa} computed the Ext-groups between the tensor power functors of the abelianization functor.
Let $\mathfrak{a}^{\#}:\kgrop\to \KMod$ denote the dual of the abelianization functor and $S^{\lambda}:\KMod\to \KMod$ the Schur functor corresponding to a partition $\lambda$.
The first Ext-groups in $\kgropMod$ between Schur functors $S^{\lambda}\circ \mathfrak{a}^{\#}$ are obtained from \cite[Corollary 18.23]{Powell--Vespa}, which generalizes a partial result extracted from \cite[Theorem 4.2]{Vespa}.

\begin{lemma}[{\cite[Corollary 18.23]{Powell--Vespa}}, see Lemma \ref{extgrop}]
 Let $\lambda,\mu$ be partitions and set $n=|\lambda|$ and $m=|\mu|$.
 Then we have
     \begin{gather*}
           \dim_{\K} \Ext^1_{\kgropMod}(
            S^{\lambda}\circ\mathfrak{a}^{\#}, S^{\mu}\circ\mathfrak{a}^{\#})
            ={\begin{cases}
                \sum_{\rho\vdash n-2}LR^{\lambda}_{\rho,1^2} LR^{\mu}_{\rho,1} & \text{if } m=n-1\\
                0 & \text{otherwise}
            \end{cases}},
    \end{gather*}
    where $LR^{\bullet}_{\bullet,\bullet}$ denotes the Littlewood--Richardson coefficient.
\end{lemma}

Schur functors $S^{\lambda}\circ \mathfrak{a}^{\#}$ are polynomial (and thus analytic), and correspond to Specht modules $S_{\lambda}$ via the equivalence of categories given by Powell.
Let $T:\kgropMod\to \AMod$ (resp. $T:\catLieMod\to \catLieCMod$) denote the functor induced by the projection $\A\to \kgrop$ (resp. $\catLieC\to \catLie$) to the degree $0$ part.
Since it is easier to compute the Ext-groups in $\catLieCMod$ than to compute the Ext-groups in $\AMod$, we compute the first Ext-groups between $T(S_{\lambda})$ in $\catLieCMod$.

\begin{theorem}[see Theorem \ref{extcatLieC}]
    Let $\lambda,\mu$ be partitions and set $n=|\lambda|$ and $m=|\mu|$.
    Then we have
     \begin{gather*}
    \begin{split}
        \Ext^1_{\catLieCMod}(T(S_{\lambda}), T(S_{\mu}))
        \cong
        \begin{cases}
        S_{\mu}\otimes_{\K\gpS_{m}}\catLie(n,m)\otimes_{\K\gpS_{n}}S_{\lambda} & \text{if }m=n-1,\\
           S_{\mu}\otimes_{\K\gpS_{m}}\ub(n,m)\otimes_{\K\gpS_{n}}S_{\lambda}  & \text{if }m=n+2,\\
            0 & \text{otherwise},
        \end{cases}
    \end{split}
    \end{gather*}
    where $\ub$ denotes the upward Brauer category, and $S_{\mu}$ is considered as a right $\K\gpS_{m}$-module.
    Moreover, we have
    \begin{gather*}
    \begin{split}
        \dim_{\K}\Ext^1_{\catLieCMod}(T(S_{\lambda}), T(S_{\mu}))
        =
        \begin{cases}
        \sum_{\rho\vdash n-2}LR^{\lambda}_{\rho,1^2}LR^{\mu}_{\rho,1} & \text{if }m=n-1,\\
         LR^{\mu}_{\lambda,2} & \text{if }m=n+2,\\
            0 & \text{otherwise}.
        \end{cases}
    \end{split}
    \end{gather*}
\end{theorem}

Via the equivalence of categories given by Kim, we obtain the first Ext-groups in $\AMod$ between $T(S^{\lambda}\circ\mathfrak{a}^{\#})$.

\begin{theorem}
    [see Theorem \ref{Ext1Amod}]
 Let $\lambda,\mu$ be partitions and set $n=|\lambda|$ and $m=|\mu|$.
 Then we have
    \begin{gather*}
    \begin{split}
        \dim_{\K}\Ext^1_{\AMod}(T(S^{\lambda}\circ \mathfrak{a}^{\#}), T(S^{\mu}\circ \mathfrak{a}^{\#}))
        =
        \begin{cases}
           \sum_{\rho\vdash n-2}LR^{\lambda}_{\rho,1^2}LR^{\mu}_{\rho,1} & \text{if } m=n-1,\\
            LR^{\mu}_{\lambda,2} & \text{if } m=n+2,\\
            0 & \text{otherwise}.
        \end{cases}
    \end{split}
    \end{gather*}
\end{theorem}

We also give the direct computation of the first Ext-groups in $\AMod$ between symmetric power functors in Theorem \ref{extsym} and between exterior power functors in Theorem \ref{extLambda}.

\subsection*{Acknowledgements}
The author would like to thank Professor Geoffrey Powell for suggesting the computation of the first Ext-groups in $\catLieCMod$, and for his valuable comments throughout drafts of the present paper.
She would also like to thank Professor Kazuo Habiro and Professor Christine Vespa for helpful comments.
This work was supported by JSPS KAKENHI Grant Number 24K16916.

\section{Preliminaries}
Here we briefly recall the category $\A$ of Jacobi diagrams in handlebodies, the category $\catLieC$ for Casimir Lie algebras, and an equivalence of categories between modules over them.
We refer the readers to \cite[Sections 1.3 and 4]{Habiro--Massuyeau} and \cite[Section 3.1]{katadaAmod} for more details on the category $\A$, and \cite[Section 3]{Hinich--Vaintrob} and \cite[Section 3.2]{katadaAmod} for more details on the category $\catLieC$.

\subsection{The category $\A$ of Jacobi diagrams in handlebodies}
Let $\K$ be a field of characteristic $0$.
The category $\A$ of Jacobi diagrams in handlebodies was introduced by Habiro and Massuyeau in \cite{Habiro--Massuyeau} as the codomain of a functor, which extends the Kontsevich invariant for bottom tangles.
The objects of $\A$ are non-negative integers, and 
the hom-space $\A(m,n)$ is the $\K$-vector space spanned by ``$(m,n)$-Jacobi diagrams in handlebodies" modulo the STU relation.
See \cite{Habiro--Massuyeau} for details.
Here, the STU relation implies the AS and IHX relations, which correspond to the axioms of Lie algebras. 

The category $\A$ is characterized by a \emph{Casimir Hopf algebra} in the following sense.

\begin{lemma}[{\cite[Theorem 5.11]{Habiro--Massuyeau}}]\label{presentationofA}
    The category $\A$ is a $\K$-linear PROP which is freely generated by a Casimir Hopf algebra $(H=1,\mu,\eta,\Delta,\varepsilon,S,\tilde{c})$.
\end{lemma}

Here, a Casimir Hopf algebra $(H=1,\mu,\eta,\Delta,\varepsilon,S,\tilde{c})$ in a $\K$-linear PROP is a cocommutative Hopf algebra $(H=1,\mu,\eta,\Delta,\varepsilon,S)$ equipped with a \emph{Casimir $2$-tensor}, which is a morphism $\tilde{c}:0\to 2$ satisfying 
\begin{gather}\label{Casimir2-tensorDelta}
    (\Delta\otimes \id_1)\tilde{c}=(\id_1\otimes \eta\otimes \id_1)\tilde{c}+\eta\otimes \tilde{c}, 
\end{gather}
\begin{gather}\label{Casimir2-tensorP}
    P_{1,1}\tilde{c}=\tilde{c}, 
\end{gather}
where $P_{1,1}$ denotes the symmetry of the $\K$-linear PROP,
and
\begin{gather}\label{Casimir2-tensorAd}
    (ad\otimes ad)(\id_1\otimes P_{1,1}\otimes \id_1)(\Delta\otimes \tilde{c})=\tilde{c}\varepsilon,
\end{gather}
where 
\begin{gather*}
    ad=\mu (\mu\otimes \id_1)(\id_{2}\otimes S)(\id_1\otimes P_{1,1})(\Delta\otimes \id_1).
\end{gather*}
We can easily check the following relations
\begin{gather}\label{Casimir2-tensorDelta'}
    (\id_1\otimes \Delta)\tilde{c}=(\id_1\otimes \eta\otimes \id_1)\tilde{c}+\tilde{c} \otimes \eta,
\end{gather}
and
\begin{gather}\label{epsilonc}
    (\varepsilon\otimes \id_1)\tilde{c}=0=(\id_1\otimes \varepsilon)\tilde{c}.
\end{gather}

The category $\A$ has an $\N$-grading, where the degree of the generating morphisms $\mu,\eta,\Delta,\varepsilon,S$ in $\A$ is $0$ and that of $\tilde{c}$ is $1$.
Let $\A_d(m,n)$ denote the degree $d$ part of the hom-space $\A(m,n)$.
The degree $0$ part $\A_0$ of the category $\A$ forms a subcategory of $\A$, which is freely generated by the cocommutative Hopf algebra.
As a $\K$-linear PROP, the category $\A_0$ is isomorphic to the $\K$-linearization $\kgrop$ of the opposite $\gr^{\op}$ of the category $\gr$ of finitely generated free groups.

In what follows, we recall a factorization of an element of $\A_d(m,n)$.  
Define $\mu^{[i]}:i\to 1$ inductively by 
$\mu^{[0]}=\eta,\; \mu^{[1]}=\id_1,\; \mu^{[i]}=\mu (\mu^{[i-1]}\otimes \id_1)$ for $i\ge 2$.
Let 
\begin{gather*}\label{mumulti}
\mu^{[p_1,\cdots,p_n]}=\mu^{[p_1]}\otimes\cdots\otimes\mu^{[p_n]}: p_1+\cdots +p_n\to n.
\end{gather*}
Similarly, we define $\Delta^{[i]}:1\to i$ by $\Delta^{[0]}=\varepsilon, \;\Delta^{[1]}=\id_1,\; \Delta^{[i]}=(\Delta^{[i-1]}\otimes \id_1)\Delta$ for $i\ge 2$, and let
\begin{gather*}
\Delta^{[q_1,\cdots,q_m]}=\Delta^{[q_1]}\otimes \cdots\otimes\Delta^{[q_m]} :m\to q_1+\cdots+q_m.
\end{gather*}
Define a homomorphism 
\begin{gather*}\label{permutation}
    P:\gpS_m\to \gr^{\op}(m,m),\quad \sigma\mapsto P_{\sigma}
\end{gather*}
by $P_{(i,i+1)}=\id_{i-1}\otimes P_{1,1}\otimes \id_{m-i-1}$ for $1\le i\le m-1$, and regard $P_{\sigma}$ as a morphism of the category $\A_0$.

\begin{lemma}[{\cite[Lemma 5.16]{Habiro--Massuyeau}}]\label{decompositionofA}
    Any element of $\A_d(m,n)$ is a linear combination of morphisms of the form
    \begin{gather*}
        \mu^{[p_1,\cdots,p_n]}P_{\sigma}(S^{e_1}\otimes\cdots\otimes S^{e_s}\otimes \id_{2d})(\Delta^{[q_1,\cdots,q_m]}\otimes \tilde{c}^{\otimes d}),
    \end{gather*}
    where $s,q_1,\cdots, q_m,p_1,\cdots,p_n\ge 0$, 
    $s=p_1+\cdots+p_n-2d=q_1+\cdots+q_m$, $\sigma\in \gpS_{s+2d}, e_1,\cdots,e_s\in \{0,1\}$.
\end{lemma}

\subsection{The $\K$-linear PROP $\catLieC$ for Casimir Lie algebras}

The $\K$-linear PROP $\catLieC$ for Casimir Lie algebras was introduced by Hinich and Vaintrob in \cite{Hinich--Vaintrob}.
Here, a \emph{Casimir Lie algebra} $(L=1,[,],c)$ in a $\K$-linear PROP is a Lie algebra $(L=1,[,])$ equipped with a \emph{Casimir element}, which is a morphism $c:0\to 2$ satisfying
\begin{gather}\label{Casimirrelation}
    P_{1,1}c=c, \quad ([,]\otimes \id_1)(\id_1\otimes c)=(\id_1\otimes [,])(c\otimes \id_1).
\end{gather}

Let $\catLie$ denote the $\K$-linear PROP for Lie algebras, which is the $\K$-linear PROP freely generated by a Lie algebra $(L=1,[,])$ whose hom-space is
\begin{gather*}
      \catLie(m,n)=\bigoplus_{f: \{1,\dots,m\}\twoheadrightarrow \{1,\dots,n\}} \bigotimes_{i=1}^{n} \Lie(|f^{-1}(i)|),
\end{gather*}
where $\Lie$ denotes the Lie operad.
The category $\catLieC$ is the $\K$-linear PROP generated by the $\K$-linear PROP $\catLie$ and a symmetric element $c\in \catLieC(0,2)$ satisfying 
\begin{gather*}
    (\id_1\otimes f)(c\otimes \id_{n-1})=(f\tau \otimes \id_1)(\id_{n-1}\otimes c)
\end{gather*}
for $f\in \Lie(n)$, where $\tau=(012\cdots n)\in \gpS_{1+n}$ acts on the operad $\Lie$, which is considered as a cyclic operad (see \cite[Section 3.1.6]{Hinich--Vaintrob} for details).
Therefore, $\catLieC$ is the $\K$-linear PROP whose objects are generated by $L=1$ and whose morphisms are generated by $[,]:2\to 1$ and $c: 0\to 2$ with relations generated by the AS, IHX-relations and the relations \eqref{Casimirrelation}.

The category $\catLieC$ has an $\N$-grading given by the number of copies of the Casimir element $c$.
Let $\catLieC(m,n)_d$ denote the degree $d$ part of the hom-space $\catLieC(m,n)$.
Then the degree $0$ part $\catLieC(m,n)_0$ forms a subcategory $(\catLieC)_0$ of $\catLieC$, which is isomorphic to $\catLie$.

The category $\catLieC$ is related to the \emph{upward Brauer category} $\ub$, which is a $\K$-linear PROP whose hom-space $\ub(m,n)$ is spanned by partitionings of $\{1^+,\cdots,m^+\}\sqcup\{1^-,\cdots,n^-\}$ into $(m+n)/2$ unordered pairs in such a way that partitionings include no pairs of two positive elements \cite{Sam--Snowden, Kupers--Randal-Williams}.
Let $\mathbb{S}$ denote the $\K$-linear category whose objects are non-negative integers and whose hom-space is $\mathbb{S}(m,n)=\K\gpS_m$ if $m=n$ and $\mathbb{S}(m,n)=0$ otherwise.
Then we have a $\K$-linear category $\catLie\otimes_{\mathbb{S}}\ub$.
The assignment of the Casimir element $c$ to each unordered pair of two negative elements induces a $\K$-linear full functor
\begin{gather*}\label{catLieuB}
    \catLie\otimes_{\mathbb{S}}\ub\to \catLieC.
\end{gather*}
In other words, any element of $\catLieC(m,n)_d$ can be decomposed into a linear combination of morphisms of the form 
\begin{gather*}
    f(\id_{L^{\otimes m}}\otimes c^{\otimes d}),
\end{gather*}
where $f\in \catLie(m+2d,n)$.

\subsection{Equivalence of categories}
Let $\KMod$ denote the category of $\K$-vector spaces and $\K$-linear maps.
For a $\K$-linear PROP $\calC$, by a (left) $\calC$-module, we mean a $\K$-linear functor from $\calC$ to $\KMod$.
Let $\calC\textbf{-}\mathrm{Mod}$ denote the category of $\calC$-modules.

In \cite{Powellanalytic}, Powell constructed a $(\kgrop, \catLie)$-bimodule ${}_{\Delta}\catAss$ and proved the following.

\begin{proposition}[{\cite[Theorem 9.19]{Powellanalytic}}]\label{adjunctionPowell}
We have the hom-tensor adjunction
\begin{gather*}
 {}_{\Delta}\catAss\otimes_{\catLie} - : \adj{\catLieMod}{\kgropMod}{}{} : \Hom_{\kgropMod}({}_{\Delta}\catAss,-).
\end{gather*}
Moreover, the functor ${}_{\Delta}\catAss\otimes_{\catLie} - $ is an equivalence of categories
\begin{gather*}
    \catLieMod \simeq \kgropMod^{\omega},
\end{gather*}
where $\kgropMod^{\omega}$ denotes the full subcategory of $\kgropMod$ whose objects correspond to analytic functors on $\gr^{\op}$.
\end{proposition}

In \cite{katadaAmod}, the author reconstructed the $(\kgrop, \catLie)$-bimodule ${}_{\Delta}\catAss$ by using the degree $0$ part $\A^{L}_0(L^{\otimes -},H^{\otimes -})$ of the hom-space $\A^{L}(L^{\otimes -},H^{\otimes -})$ of the category $\A^{L}$ of extended Jacobi diagrams in handlebodies. (See \cite[Sections 4.2, 4.3]{katada2} and \cite[Section 3.3]{katadaAmod} for details on the category $\A^{L}$.)
The hom-space $\A^L(L^{\otimes -},H^{\otimes -})$ has a canonical structure of an $(\A,\catLieC)$-bimodule, which yields a generalization of the hom-tensor adjunction in Proposition \ref{adjunctionPowell} as follows.

\begin{proposition}[{\cite[Theorem 4.4]{katadaAmod}}]\label{adjunctionA}
We have an $(\A,\catLieC)$-bimodule $\A^{L}(L^{\otimes -},H^{\otimes -})$, which induces an adjunction
\begin{gather*}
    \A^{L}(L^{\otimes -},H^{\otimes -})\otimes_{\catLieC} - : \adj{\catLieCMod}{\AMod}{}{} : \Hom_{\AMod}(\A^{L}(L^{\otimes -},H^{\otimes -}),-).
\end{gather*}
\end{proposition}

The inclusion functor
\begin{gather*}
   \kgrop\hookrightarrow \A
\end{gather*}
induces a forgetful functor 
\begin{gather*}\label{UforA}
   U: \AMod\to \kgropMod,
\end{gather*}
which is exact and has a left adjoint.
We define a full subcategory $\AMod^{\omega}\subset \AMod$ of \emph{analytic $\A$-modules}, where an $\A$-module $M$ is analytic if $U(M)\in \kgropMod^{\omega}$.
Then the forgetful functor $U$ restricts to analytic functors, that is, we have a $\K$-linear functor
\begin{gather}\label{UforAomega}
    U:\AMod^{\omega}\to \kgropMod^{\omega}.
\end{gather}
The adjunction in Proposition \ref{adjunctionA} restricts to analytic functors.

\begin{proposition}[{\cite[Theorem 4.9]{katadaAmod}}]\label{adjunctionAomega}
The $\A$-module $\A^{L}(L^{\otimes -}, H^{\otimes -})$ is analytic, and
the adjunction in Theorem \ref{adjunctionA} restricts to an adjunction 
\begin{gather*}
    \A^{L}(L^{\otimes -}, H^{\otimes -})\otimes_{\catLieC} - : \adj{\catLieCMod}{\AMod^{\omega}}{}{} : \Hom_{\AMod^{\omega}}(\A^{L}(L^{\otimes -}, H^{\otimes -}),-).
\end{gather*}    
\end{proposition}

Kim \cite{Kimtalk} has independently obtained an adjunction which is equivalent to the adjunction in Proposition \ref{adjunctionAomega}, and furthermore, he has proved that the adjunction gives an equivalence of categories. (See \cite[Remark 4.10]{katadaAmod}.)

\begin{proposition}[Kim \cite{Kimtalk}]\label{equivalenceKim}
The functor $ \A^{L}(L^{\otimes -}, H^{\otimes -})\otimes_{\catLieC} - $ is an equivalence of categories
\begin{gather*}
        \catLieCMod\simeq \AMod^{\omega}.
\end{gather*}
\end{proposition}

The inclusion functor
\begin{gather*}
   \catLie\hookrightarrow \catLieC
\end{gather*}
induces a forgetful functor
\begin{gather}\label{UforcatLieC}
    U:\catLieCMod\to \catLieMod,
\end{gather}
which is exact and has a left adjoint.
The forgetful functor $U$ defined in \eqref{UforcatLieC} corresponds to the forgetful functor $U$ defined in \eqref{UforAomega} via the equivalences of categories in Propositions \ref{adjunctionPowell} and \ref{equivalenceKim}.

By the $\N$-grading of the $\K$-linear category $\A$, we have a projection functor
\begin{gather*}
    \A\twoheadrightarrow \A_0\cong \kgrop,
\end{gather*}
which induces a $\K$-linear functor
\begin{gather*}\label{TforA}
    T: \kgropMod\to \AMod,
\end{gather*}
which is fully-faithful and exact.
Since we have $U\circ T=\id_{\kgropMod}$, the functor $T$ restricts to analytic functors, that is, we have a $\K$-linear functor
\begin{gather}\label{TforAomega}
    T: \kgropMod^{\omega}\to \AMod^{\omega}.
\end{gather}
Similarly, the projection to the degree $0$ part induces a $\K$-linear functor
\begin{gather}\label{TforcatLieC}
     T: \catLieMod\to \catLieCMod.
\end{gather}
The functor $T$ defined in \eqref{TforcatLieC} corresponds to the functor $T$ defined in \eqref{TforAomega} via the equivalences of categories in Proposition  \ref{adjunctionPowell} and \ref{equivalenceKim}.

\subsection{Specht modules and Schur functors}
Here we recall Specht modules and Schur functors.
We refer the reader to \cite{Fulton--Harris} on representation theory of symmetric groups.

For a partition $\lambda$ of a non-negative integer, let $l(\lambda)$ denote the \emph{length} of $\lambda$ and $|\lambda|$ the \emph{size} of $\lambda$. We write $\lambda\vdash |\lambda|$.
Let $c_{\lambda}\in \K\gpS_{|\lambda|}$ denote the \emph{Young symmetrizer} corresponding to $\lambda$, that is, 
$$c_{\lambda}=\left(\sum_{\sigma \in R_{\lambda}}\sigma \right)\left(\sum_{\tau\in C_{\lambda}}\sgn(\tau)\tau\right),$$
where $R_{\lambda}$ is the row stabilizer and $C_{\lambda}$ is the column stabilizer of the canonical tableau of $\lambda$. (See \cite{Fulton--Harris} for details.) 
The \emph{Specht module} $S_{\lambda}$ corresponding to $\lambda\vdash d$ is defined by
\begin{gather*}
    S_{\lambda}=\K\gpS_{d}\cdot c_{\lambda}.
\end{gather*}
It is well known that $\{S_{\lambda}\mid \lambda\vdash d\}$ is the set of isomorphism classes of irreducible representations of $\gpS_{d}$.
Since $\catLie(m,n)=0$ for $m<n$, the set of isomorphism classes of simple $\catLie$-modules is $\{S_{\lambda}\mid d\ge 0, \lambda\vdash d\}$.

The \emph{Schur functor}
$$S^{\lambda}:\KMod\to\KMod$$
corresponding to $\lambda\vdash d$ is defined for $V\in \Ob(\KMod)$ by
$$S^{\lambda}(V)=V^{\otimes d}\otimes_{\K\gpS_{d}} S_{\lambda}.$$
The $\GL(V)$-module $S^{\lambda}(V)$ is irreducible if $\dim(V)\ge l(\lambda)$ and otherwise we have $S^{\lambda}(V)=0$.

Let 
$$\mathfrak{a}:\gr\to \KMod$$
denote the abelianization functor and
$$\mathfrak{a}^{\#}=\Hom(-^{\ab},\K): \gr^{\op}\to \KMod$$
the dual of it.
By abuse of notation, let $\mathfrak{a}^{\#}$ denote the functor from $\kgrop$.
Then the $\catLie$-module $S_{\lambda}$ corresponds to the $\kgrop$-module $S^{\lambda}\circ \mathfrak{a}^{\#}$ via the equivalence of categories $\A^{L}_0(L^{\otimes -},H^{\otimes -})\otimes_{\catLie}-$ in Proposition \ref{adjunctionPowell}, that is, we have
\begin{gather*}
     \A^{L}_0(L^{\otimes -},H^{\otimes -})\otimes_{\catLie}S_{\lambda}\cong S^{\lambda}\circ \mathfrak{a}^{\#}.
\end{gather*}
Therefore, $\{S^{\lambda}\circ \mathfrak{a}^{\#}\mid d\ge 0, \lambda\vdash d\}$ is the set of isomorphism classes of simple objects in $\kgropMod^{\omega}$.

\section{The first Ext-groups in the category of $\A$-modules}\label{sectionExtgp}

In this section, we study the Ext-groups in $\AMod$ between simple $\A$-modules which are induced by Schur functors.
The $0$-th Ext-group $\Ext^0_{\AMod}(T(S^{\lambda}\circ\mathfrak{a}^{\#}),T(S^{\mu}\circ\mathfrak{a}^{\#}))$ between two $\A$-modules $T(S^{\lambda}\circ\mathfrak{a}^{\#})$ and $T(S^{\mu}\circ\mathfrak{a}^{\#})$ is computed as follows.

\begin{lemma}
We have
\begin{gather*}
    \begin{split}
        \Ext^0_{\AMod}(T(S^{\lambda}\circ\mathfrak{a}^{\#}),T(S^{\mu}\circ\mathfrak{a}^{\#}))
       \cong  
        \begin{cases}
                \K  & \text{if } \lambda=\mu\\
                0 & \text{otherwise}.
        \end{cases}
    \end{split}
\end{gather*}
\end{lemma}

\begin{proof}
Since the functor $T$ is fully-faithful, we have
\begin{gather*}
    \begin{split}
        \Ext^0_{\AMod}(T(S^{\lambda}\circ\mathfrak{a}^{\#}),T(S^{\mu}\circ\mathfrak{a}^{\#}))
        &\cong \Hom_{\AMod}(T(S^{\lambda}\circ\mathfrak{a}^{\#}),T(S^{\mu}\circ\mathfrak{a}^{\#}))\\
        &\cong \Hom_{\kgropMod}(S^{\lambda}\circ\mathfrak{a}^{\#},S^{\mu}\circ\mathfrak{a}^{\#})\\
        &\cong \Hom_{\kgropMod^{\omega}}(S^{\lambda}\circ\mathfrak{a}^{\#},S^{\mu}\circ\mathfrak{a}^{\#})\\
        &\cong \Hom_{\catLieMod}(S_{\lambda},S_{\mu})\\
       &\cong \Hom_{\mathbb{S}\textbf{-}\mathrm{Mod}} (S_{\lambda},S_{\mu})
    \end{split}
\end{gather*}
by the equivalence of categories in Proposition \ref{adjunctionPowell}.
Therefore, the statement follows from Schur's lemma.
\end{proof}

In what follows, we compute the first Ext-groups in $\AMod$ between $T(S^{\lambda}\circ\mathfrak{a}^{\#})$ and $T(S^{\mu}\circ\mathfrak{a}^{\#})$.

\subsection{The Ext-groups in $\kgropMod$}
Vespa \cite{Vespa} computed the Ext-groups between tensor power functors $T^d\circ \mathfrak{a}$, symmetric power functors $S^d\circ \mathfrak{a}$ and exterior power functors $\Lambda^d\circ \mathfrak{a}$ in the functor category $\operatorname{Funct}(\gr,\KMod)$.

The category of polynomial functors is known to be thick, that is, closed under quotients, subobjects and extensions (see \cite[Proposition 3.7]{Hartl--Pirashvili--Vespa}).
Therefore, the inclusion $\kgropMod^{\omega}\hookrightarrow \kgropMod$ induces a $\K$-linear isomorphism
\begin{gather}\label{extensionofSchurfunctors}
    \Ext^1_{\kgropMod^{\omega}}(
            S^{\lambda}\circ\mathfrak{a}^{\#}, S^{\mu}\circ\mathfrak{a}^{\#})\cong \Ext^1_{\kgropMod}(
            S^{\lambda}\circ\mathfrak{a}^{\#}, S^{\mu}\circ\mathfrak{a}^{\#}).
\end{gather}
More generally, it follows from \cite[Th\'{e}or\`{e}me 1]{Djament--Pirashvili--Vespa} that the Ext-groups between polynomial functors in $\kgropMod$ are isomorphic to the Ext-groups in $\kgropMod^{\omega}$.

The first Ext-groups in $\kgropMod$ between Schur functors are obtained from \cite[Corollary 18.23]{Powell--Vespa}, which generalizes a partial result extracted from \cite[Theorem 4.2]{Vespa}.

\begin{lemma}\label{extgrop}
We have
     \begin{gather*}
            \dim_{\K}\Ext^1_{\kgropMod^{\omega}}(
            S^{\lambda}\circ\mathfrak{a}^{\#}, S^{\mu}\circ\mathfrak{a}^{\#})
            ={\begin{cases}
                \sum_{\rho\vdash |\lambda|-2}LR^{\lambda}_{\rho,1^2} LR^{\mu}_{\rho,1} & \text{if } |\mu|=|\lambda|-1\\
                0 & \text{otherwise}
            \end{cases}},
    \end{gather*}
    where $LR^{\bullet}_{\bullet,\bullet}$ denotes the Littlewood--Richardson coefficient.
\end{lemma}

\begin{proof}
Since $\K$-linearization induces an isomorphism
$$\operatorname{Funct}(\gr^{\op},\KMod)\xrightarrow{\cong} \kgropMod,$$
we have
\begin{gather*}
    \begin{split}
        \Ext^1_{\kgropMod}(
            S^{\lambda}\circ\mathfrak{a}^{\#}, S^{\mu}\circ\mathfrak{a}^{\#})
        &\cong 
        \Ext^1_{\operatorname{Funct}(\gr^{\op},\KMod)}(
            S^{\lambda}\circ\mathfrak{a}^{\#}, S^{\mu}\circ\mathfrak{a}^{\#})  \\
        &\cong \Ext^1_{\operatorname{Funct}(\gr,\KMod)}(
            S^{\mu}\circ\mathfrak{a}, S^{\lambda}\circ\mathfrak{a})
    \end{split}
\end{gather*}
by duality.
Then the statement follows from \eqref{extensionofSchurfunctors} and \cite[Corollary 18.23]{Powell--Vespa}.
\end{proof}

\subsection{The Yoneda Ext}
The category $\AMod$ is an abelian category with enough projectives.
For $\A$-modules $M$ and $M'$, the Ext-group $\Ext^*_{\AMod}(M,M')$ is defined to be the cohomology of the cochain complex $\Hom_{\AMod}(P_*,M')$, where $P_*$ is a projective resolution of $M$.
It is well known that the Ext-group is equivalent to the Yoneda Ext.
The Yoneda Ext $\Ext^1(M,M')$ of degree $1$ for $\A$-modules $M$ and $M'$ is defined to be the collection of equivalence classes of extensions, where an extension of $M$ by $M'$ is a short exact sequence
$$
0\to M'\to X\to M\to 0
$$
in $\AMod$, and where two extensions are equivalent if there is a commutative diagram
\begin{gather*}
    \xymatrix{
    0\ar[r]&M'\ar[r]\ar@{=}[d]&X\ar[r]\ar[d]^{\cong}&M\ar[r]\ar@{=}[d]&0\\
     0\ar[r]&M'\ar[r]&X'\ar[r]&M\ar[r]&0.
    }
\end{gather*}
Then there is a one-to-one correspondence
\begin{gather*}
    \Ext^1_{\AMod}(M,M')\cong \Ext^1(M,M')
\end{gather*}
(see \cite[Theorem 3.4.3]{Weibel} for example).

Since the forgetful functor $U$ and the functor $T$ are exact, we have a $\K$-linear map 
\begin{gather*}
    \pr: \Ext^1_{\AMod}(M,M')\to \Ext^1_{\kgropMod}(U(M),U(M'))
\end{gather*}
for $\A$-modules $M$ and $M'$, and a $\K$-linear map
\begin{gather*}
    \iota: \Ext^1_{\kgropMod}(N,N')\to \Ext^1_{\AMod}(T(N),T(N'))
\end{gather*}
for $\kgrop$-modules $N$ and $N'$. Since we have $\pr\circ \iota=\id_{\Ext^1_{\kgropMod}(N,N')}$,
the $\K$-linear map $\iota$ is injective.

In the rest of this section, we will compute the Yoneda Ext of degree $1$ for the $\A$-modules $T(S^{\lambda}\circ \mathfrak{a}^{\#})$ and $T(S^{\mu}\circ \mathfrak{a}^{\#})$.
Since the forgetful functor $U$ is exact and we have $U\circ T=\id_{\kgropMod}$, it follows from \eqref{extensionofSchurfunctors} that we have
\begin{gather*}
    \Ext^1_{\AMod}(T(S^{\lambda}\circ \mathfrak{a}^{\#}), T(S^{\mu}\circ \mathfrak{a}^{\#}))\cong  \Ext^1_{\AMod^{\omega}}(T(S^{\lambda}\circ \mathfrak{a}^{\#}), T(S^{\mu}\circ \mathfrak{a}^{\#})).
\end{gather*}

\begin{remark}
 Via the equivalence of categories between $\catLieCMod$ and $\AMod^{\omega}$ in Proposition \ref{equivalenceKim}, we have
 \begin{gather*}
     \Ext^1_{\AMod^{\omega}}(T(S^{\lambda}\circ \mathfrak{a}^{\#}), T(S^{\mu}\circ \mathfrak{a}^{\#}))\cong \Ext^1_{\catLieCMod}(T(S_{\lambda}), T(S_{\mu})).
 \end{gather*}
 We also have a $\K$-linear map
 \begin{gather*}
     \pr: \Ext^1_{\catLieCMod}(K,K')\to \Ext^1_{\catLieMod}(U(K),U(K'))
 \end{gather*}
 induced by the forgetful functor $U$ defined in \eqref{UforcatLieC} and a $\K$-linear map
 \begin{gather*}
     \iota: \Ext^1_{\catLieMod}(J,J')\to \Ext^1_{\catLieCMod}(T(J),T(J'))
 \end{gather*}
 induced by the functor $T$ defined in \eqref{TforcatLieC}, which satisfy $\pr\circ \iota=\id_{\Ext^1_{\catLieMod}(J,J')}$.
\end{remark}

\subsection{The first Ext-groups in $\catLieCMod$}
Here, we compute the first Ext-groups between simple objects in $\catLieCMod$ induced by Specht modules.
The first version of the draft of the present paper included the direct computation of the first Ext-groups between Schur functors in $\AMod$ in some cases, and left other cases as a conjecture.
After the author sent a draft of this paper to Geoffrey Powell, he suggested the following theorem, which implies that her conjecture is immediate in the case of $\catLieCMod$.

\begin{theorem}\label{extcatLieC}
    Let $\lambda,\mu$ be partitions and set $n=|\lambda|$ and $m=|\mu|$.
    Then we have
    \begin{gather*}
    \begin{split}
        \Ext^1_{\catLieCMod}(T(S_{\lambda}), T(S_{\mu}))
        \cong
        \begin{cases}
        S_{\mu}\otimes_{\K\gpS_{m}}\catLie(n,m)\otimes_{\K\gpS_{n}}S_{\lambda} & \text{if }m=n-1,\\
           S_{\mu}\otimes_{\K\gpS_{m}}\ub(n,m)\otimes_{\K\gpS_{n}}S_{\lambda}  & \text{if }m=n+2,\\
            0 & \text{otherwise},
        \end{cases}
    \end{split}
    \end{gather*}
    where $S_{\mu}$ is considered as a right $\K\gpS_{m}$-module.
    Moreover, we have
    \begin{gather*}
    \begin{split}
        \dim_{\K}\Ext^1_{\catLieCMod}(T(S_{\lambda}), T(S_{\mu}))
        =
        \begin{cases}
           \sum_{\rho\vdash n-2}LR^{\lambda}_{\rho,1^2}LR^{\mu}_{\rho,1} & \text{if }m=n-1,\\
           LR^{\mu}_{\lambda,2} & \text{if }m=n+2,\\
            0 & \text{otherwise}.
        \end{cases}
    \end{split}
    \end{gather*}
\end{theorem}

\begin{proof}
Let $K$ be an extension of $T(S_{\lambda})$ by $T(S_{\mu})$ in $\catLieCMod$,
that is, we have a short exact sequence of $\catLieC$-modules
\begin{gather*}
    0\to T(S_{\mu})\to K\to T(S_{\lambda})\to 0.
\end{gather*}
Then we have $K(n)=S_{\lambda}, K(m)=S_{\mu}, K(p)=0$ for $p\neq m,n$. Moreover, we have $K(f)=0$ for any morphism $f\in \catLieC(p,q)$ unless $(p,q)=(n,m)$.

Note that any morphism $f\in \catLieC(p,q)_d$ with $d\ge 1$ can be decomposed into a linear combination of morphisms of the form
$$g(\id_{p}\otimes c^{\otimes d})=g(\id_{p+2}\otimes c^{\otimes d-1})(\id_p\otimes c)$$
for some $g\in \catLie(p+2d,q)$, and any morphism $f\in \catLieC(p,q)_0=\catLie(p,q)$ can be decomposed into a linear combination of morphisms of the form
$$f=g(\id_{i}\otimes [,]\otimes \id_{p-i-2})$$
for some $i\in \{0,\cdots,p-2\}$ and $g\in \catLie(p-1,q)$.

If $m\notin \{n-1,n+2\}$, then we have $K(\id_n\otimes c)=0$ and $K(\id_{i}\otimes [,]\otimes \id_{n-i-2})=0$ for any $i\in \{0,\cdots,n-2\}$, 
and thus we have $K(f)=0$ for any $f\in \catLieC(n,m)$.
Therefore, we have $\Ext^1_{\catLieCMod}(T(S_{\lambda}), T(S_{\mu}))=0$.

If $m=n-1$, then we have $K(\id_n\otimes c)=0$ and thus $K(f)=0$ for any $f\in \catLieC(p,q)_{d\ge 1}$, which implies that $K$ is the trivial $\catLieC$-module $TU(K)$ induced by an extension $U(K)$ of $S_{\lambda}$ by $S_{\mu}$ in $\catLieMod$.
Therefore, we have 
$$\Ext^1_{\catLieCMod}(T(S_{\lambda}), T(S_{\mu}))\cong \Ext^1_{\catLieMod}(S_{\lambda}, S_{\mu}).$$
Since $K$ is determined by $K(f)$ for $f\in \catLie(n,n-1)$ which is compatible with the $\K\gpS_n$-module structure of $K(n)=S_{\lambda}$ and the $\K\gpS_{n-1}$-module structure of $K(n-1)=S_{\mu}$, we obtain
$$\Ext^1_{\catLieMod}(S_{\lambda}, S_{\mu})\cong S_{\mu}\otimes_{\K\gpS_{n-1}}\catLie(n,n-1) \otimes_{\K\gpS_{n}} S_{\lambda},$$
where we consider $S_{\mu}$ as a right $\K\gpS_{n-1}$-module.
It follows from Proposition \ref{adjunctionPowell} and Lemma \ref{extgrop} that 
\begin{gather*}
    \begin{split}
\dim_{\K}\Ext^1_{\catLieMod}(S_{\lambda}, S_{\mu})
&=\dim_{\K}\Ext^1_{\kgropMod^{\omega}}(S^{\lambda}\circ \mathfrak{a}^{\#}, S^{\mu}\circ \mathfrak{a}^{\#})\\
&= \sum_{\rho\vdash n-2}LR^{\lambda}_{\rho,1^2}LR^{\mu}_{\rho,1}.
    \end{split}
\end{gather*}

If $m=n+2$, then we have $K(f)=0$ for any $f\in \catLieC(p,q)$ unless $f\in \catLieC(n,n+2)_1\cong \ub(n,n+2)$.
Therefore, $K$ is determined by $K(f)$ for $f\in \ub(n,n+2)$ which is compatible with the $\K\gpS_n$-module structure of $K(n)=S_{\lambda}$ and the $\K\gpS_{n+2}$-module structure of $K(n+2)=S_{\mu}$, and thus we obtain
\begin{gather*}
    \Ext^1_{\catLieCMod}(T(S_{\lambda}), T(S_{\mu}))\cong S_{\mu}\otimes_{\K\gpS_{n+2}}\ub(n,n+2) \otimes_{\K\gpS_{n}} S_{\lambda}.
\end{gather*}
The hom-space $\ub(n,n+2)$ is generated by $\id_{n}\otimes c$ as a left $\K\gpS_{n+2}$-module.
Therefore, we have
\begin{gather*}
\begin{split}
    S_{\mu}\otimes_{\K\gpS_{n+2}}\ub(n,n+2) \otimes_{\K\gpS_{n}} S_{\lambda}
    &\cong c_{\mu}\K\gpS_{n+2}\otimes_{\K\gpS_{n+2}}\K\gpS_{n+2}\langle \id_n\otimes c\rangle \otimes_{\K\gpS_{n}}\K\gpS_{n}c_{\lambda}\\
    &\cong   c_{\mu}\K\gpS_{n+2}(c_{\lambda}\otimes c_{(2)}).
\end{split}
\end{gather*}
By the basic fact of representation theory of symmetric groups,
we have
\begin{gather*}
    \dim_{\K}(c_{\mu}\K \gpS_{n+2}(c_{\lambda}\otimes c_{(2)}))=LR^{\mu}_{\lambda,2},
\end{gather*}
which completes the proof.
\end{proof}

Via the equivalence of categories in Proposition \ref{equivalenceKim}, by Theorem \ref{extcatLieC}, we obtain the first Ext-groups between simple objects in $\AMod^{\omega}$ induced by Schur functors.

\begin{theorem}\label{Ext1Amod}
 Let $\lambda,\mu$ be partitions and set $n=|\lambda|$ and $m=|\mu|$.
 Then we have
    \begin{gather*}
    \begin{split}
        \dim_{\K}\Ext^1_{\AMod}(T(S^{\lambda}\circ \mathfrak{a}^{\#}), T(S^{\mu}\circ \mathfrak{a}^{\#}))
        =
        \begin{cases}
        \sum_{\rho\vdash n-2}LR^{\lambda}_{\rho,1^2}LR^{\mu}_{\rho,1} & \text{if } m=n-1,\\
         LR^{\mu}_{\lambda,2} & \text{if } m=n+2,\\
            0 & \text{otherwise}.
        \end{cases}
    \end{split}
    \end{gather*}
\end{theorem}

\subsection{The computation in $\AMod$}

The computation of extensions in $\AMod$ is more complicated than that in $\catLieCMod$.
In the rest of this section, we give the direct computation of the Yoneda Ext of degree $1$ for symmetric power functors and exterior power functors in $\AMod$.
In the computation, we use the following lemma.

\begin{lemma}\label{extensionAB}
Let $N,N'$ be $\kgrop$-modules.
Let $F$ be an extension of $T(N)$ by $T(N')$ in $\AMod$.
Then $F$ satisfies the following properties:
\begin{itemize}
    \item $U(F)$ is an extension of $N$ by $N'$ in $\kgropMod$,
    in particular,
    for $m\ge 0$, we have a $\K$-linear isomorphism   
    $$F(m)\cong N(m)\oplus N'(m),$$
    and for a morphism $f\in \A_0(m,n)$, we have 
    $$F(f)((x,y))= (N(f)(x), F(f)(x)-N(f)(x)+N'(f)(y))$$
    for $(x,y)\in N(m)\oplus N'(m)\cong F(m)$,
    \item for a morphism $f\in \A_1(m,n)$, we have 
    $$F(f)((x,y))= (0,F(f)(x))$$ 
    for $(x,y)\in N(m)\oplus N'(m)\cong F(m)$,
    \item for a morphism $f\in \A_{\ge 2}(m,n)$, we have $F(f)=0$.
\end{itemize}
Therefore, for any morphisms $f_1\in \A_1(n,p)$ and $g_0\in \A_0(m,n)$, we have
\begin{gather}\label{F=Ag0}
    F(f_1\circ g_0)(x)=F(f_1)(N(g_0)(x))\in N'(p)
\end{gather}
for $x\in N(m)$.
\end{lemma}

\begin{proof}
    For a morphism $f\in \A_0(m,n)$, it follows from
    $$F(f)|_{N'(m)}:N'(m)\xrightarrow{T(N')(f)}N'(n)\hookrightarrow F(n)$$
    that we have
    $F(f)(y)=(0,N'(f)(y))$ for $y\in N'(m)$,
    and it follows from
    $$T(N)(f):N(m)\xrightarrow{F(f)\mid_{N(m)}}F(n)\twoheadrightarrow N(n)$$
    that we have $F(f)(x)=(N(f)(x),F(f)(x)-N(f)(x))$ for $x\in N(m)$.
    
    For a morphism $f\in \A_1(m,n)$, we have
    $$F(f)|_{N'(m)}:N'(m)\xrightarrow{T(N')(f)=0}N'(n)\hookrightarrow F(n)$$
    and 
    $$0=T(N)(f):N(m)\xrightarrow{F(f)\mid_{N(m)}}F(n)\twoheadrightarrow N(n).$$
    Therefore, we have $F(f)(x,y)=(0,F(f)(x))$ for $x\in N(m), y\in N'(m)$.
    It is easy to see that we have $F(f\circ g)=F(f)\circ F(g)=0$ for any composable morphisms $f\in \A_1(m,n)$ and $g\in \A_1(p,m)$.
    Therefore, we have $F(f)=0$ for $f\in \A_{\ge 2}(m,n)$ since any morphism of $\A_{\ge 2}(m,n)$ is a linear combination of morphisms which are obtained from degree $1$ morphisms by composition.

    The equation \eqref{F=Ag0} follows from
    \begin{gather*}
    \begin{split}
    F(f_1\circ g_0)(x)&=F(f_1)(F(g_0)(x))\\
    &=F(f_1)(N(g_0)(x))+F(f_1)(F(g_0)(x)-N(g_0)(x))\\
    &=F(f_1)(N(g_0)(x)).
    \end{split}
    \end{gather*}
    This completes the proof.
\end{proof}

We compute the first Ext-groups between symmetric power functors in $\AMod$.

\begin{theorem}\label{extsym}
    For $d,d'\ge 0$, we have
    \begin{gather*}
        \begin{split}
            \Ext^1_{\AMod}(T(S^{d}\circ\mathfrak{a}^{\#}), T(S^{d'}\circ\mathfrak{a}^{\#}))
            \cong\begin{cases}
                \K & \text{if }d'=d+2\\
                0 & \text{otherwise}
            \end{cases}.
        \end{split}
    \end{gather*}
\end{theorem}

\begin{proof}
Let $F$ be an extension of $T(S^{d}\circ\mathfrak{a}^{\#})$ by $T(S^{d'}\circ\mathfrak{a}^{\#})$.
It follows from Lemmas \ref{extgrop} and \ref{extensionAB} that 
$U(F)\cong S^{d}\circ\mathfrak{a}^{\#}\oplus S^{d'}\circ\mathfrak{a}^{\#}$ as $\kgrop$-modules and that $F$ is determined by the morphisms $F(f):S^{d}\circ\mathfrak{a}^{\#}(m)\to S^{d'}\circ\mathfrak{a}^{\#}(n)$ for $f\in \A_1(m,n)$.

Let $\{e_i\mid 1\le i\le d\}$ be a basis for $\Z^d$ and $\{x_i\mid 1\le i\le d\}$ the dual basis for $\Hom(\Z^d,\K)\cong \K^d$.
We take a generator 
$$x=x_1\cdots x_d\in S^{d}\circ\mathfrak{a}^{\#}(d)\cong \Sym^d(\Hom(\Z^d,\K))$$
of the simple $\kgrop$-module $S^{d}\circ\mathfrak{a}^{\#}$, and set 
$$F(\id_{d}\otimes \tilde{c})(x)=
\sum_{I=(i_1,\cdots,i_{d+2})\in \Lambda}a_I x_1^{i_1}\cdots x_{d+2}^{i_{d+2}} \in S^{d'}\circ\mathfrak{a}^{\#}(d+2),$$
where $a_I\in \K$ and 
$\Lambda=\{I=(i_1,\cdots,i_{d+2})\mid  \sum_{j=1}^{d+2}i_j=d', \; i_1,\cdots,i_{d+2}\ge 0\}$.
It follows from the relation \eqref{Casimir2-tensorDelta} of the Casimir Hopf algebra in $\A$ that
\begin{gather*}
\begin{split}
    F((\id_{d}\otimes \Delta\otimes \id_1)\circ (\id_{d}\otimes \tilde{c}))(x)
    &=F((\id_{d+1}\otimes \eta\otimes \id_1+\id_{d}\otimes \eta\otimes \id_{2})\circ (\id_{d}\otimes \tilde{c}))(x).
\end{split}
\end{gather*}
It follows that $a_I=0$ unless $i_{d+1}= 1$ since we have
\begin{gather*}
\begin{split}
F((\id_{d}\otimes \Delta\otimes \id_1)\circ (\id_{d}\otimes \tilde{c}))(x)
&=F(\id_{d}\otimes \Delta\otimes \id_1)(\sum_{I\in \Lambda}a_I x_1^{i_1}\cdots x_{d+2}^{i_{d+2}})\\
&=S^{d'}\circ\mathfrak{a}^{\#}(\id_{d}\otimes \Delta\otimes \id_1)(\sum_{I\in \Lambda}a_I x_1^{i_1}\cdots x_{d+2}^{i_{d+2}})\\
&=\sum_{I\in \Lambda}a_I x_1^{i_1}\cdots x_d^{i_d}(x_{d+1}+x_{d+2})^{i_{d+1}} x_{d+3}^{i_{d+2}},
\end{split}
\end{gather*}
and 
\begin{gather*}
\begin{split}
     &F((\id_{d+1}\otimes \eta\otimes \id_1+\id_{d}\otimes \eta\otimes \id_{2})\circ (\id_{d}\otimes \tilde{c}))(x)\\
     &= F(\id_{d+1}\otimes \eta\otimes \id_1+\id_{d}\otimes \eta\otimes \id_{2})(\sum_{I\in \Lambda}a_I x_1^{i_1}\cdots x_{d+2}^{i_{d+2}})\\
     &=S^{d'}\circ\mathfrak{a}^{\#}(\id_{d+1}\otimes \eta\otimes \id_1+\id_{d}\otimes \eta\otimes \id_{2})(\sum_{I\in \Lambda}a_I x_1^{i_1}\cdots x_{d+2}^{i_{d+2}})\\
     &=\sum_{I\in \Lambda}a_I x_1^{i_1} \cdots x_d^{i_d} (x_{d+1}^{i_{d+1}}+x_{d+2}^{i_{d+1}})x_{d+3}^{i_{d+2}}.
\end{split}
\end{gather*}
In a similar way, we have $a_I=0$ unless $i_{d+2}=1$ by the relation \eqref{Casimir2-tensorDelta'}.
For any $1\le j\le d$, we also have $a_I=0$ unless $i_j= 1$
since we have 
\begin{gather*}
    \begin{split}
    &F((\id_{j-1}\otimes \Delta\otimes \id_{d-j+2})\circ(\id_{d}\otimes \tilde{c}))(x)\\
    &=F(\id_{d+1}\otimes\tilde{c})F(\id_{j-1}\otimes \Delta\otimes \id_{d-j})(x)\\
    &=F(\id_{d+1}\otimes\tilde{c})F(\id_{j-1}\otimes \eta\otimes \id_{d-j+1}+\id_{j}\otimes \eta\otimes \id_{d-j})(x)\\
    &=F((\id_{j-1}\otimes \eta\otimes \id_{d-j+3}+\id_{j}\otimes \eta\otimes \id_{d-j+2})\circ (\id_{d}\otimes \tilde{c}))(x),
\end{split}
\end{gather*}
where the second identity follows from
\begin{gather*}
\begin{split}
     &F(\id_{d+1}\otimes\tilde{c})F(\id_{j-1}\otimes \Delta\otimes \id_{d-j})(x)\\
     &=F(\id_{d+1}\otimes\tilde{c})(S^{d}\circ\mathfrak{a}^{\#})(\id_{j-1}\otimes \Delta\otimes \id_{d-j})(x)\quad (\text{by }\eqref{F=Ag0})\\
     &=F(\id_{d+1}\otimes\tilde{c})(x_1\cdots x_{j-1} (x_{j}+x_{j+1}) x_{j+2}\cdots x_{d+1})\\
     &=F(\id_{d+1}\otimes\tilde{c})(S^{d}\circ\mathfrak{a}^{\#})(\id_{j-1}\otimes \eta\otimes \id_{d-j+1}+\id_{j}\otimes \eta\otimes \id_{d-j})(x)\\
     &=F(\id_{d+1}\otimes\tilde{c})F(\id_{j-1}\otimes \eta\otimes \id_{d-j+1}+\id_{j}\otimes \eta\otimes \id_{d-j})(x)\quad (\text{by }\eqref{F=Ag0}),
\end{split}
\end{gather*}
and since we have
\begin{gather*}
\begin{split}
    F((\id_{j-1}\otimes \Delta\otimes \id_{d-j+2})\circ(
    \id_{d}\otimes \tilde{c}))(x)
    &=F(\id_{j-1}\otimes \Delta\otimes \id_{d-j+2})(\sum_{I\in \Lambda}a_I x_1^{i_1}\cdots x_{d+2}^{i_{d+2}})\\
    &=\sum_{I\in \Lambda}a_I x_1^{i_1}\cdots x_{j-1}^{i_{j-1}} ( x_{j}+ x_{j+1})^{i_{j}}x_{j+2}^{i_{j+1}} \cdots x_{d+3}^{i_{d+2}}
\end{split}
\end{gather*}
and
\begin{gather*}
\begin{split}
   &F((\id_{j-1}\otimes \eta\otimes \id_{d-j+3}+\id_{j}\otimes \eta\otimes \id_{d-j+2})\circ (\id_{d}\otimes \tilde{c}))(x)\\
   &=F(\id_{j-1}\otimes \eta\otimes \id_{d-j+3}+\id_{j}\otimes \eta\otimes \id_{d-j+2})(\sum_{I\in \Lambda}a_I x_1^{i_1}\cdots x_{d+2}^{i_{d+2}})\\
   &=\sum_{I\in \Lambda}a_I x_1^{i_1}\cdots x_{j-1}^{i_{j-1}} ( x_{j}^{i_{j}}+ x_{j+1}^{i_{j}})x_{j+2}^{i_{j+1}} \cdots x_{d+3}^{i_{d+2}}.
\end{split}
\end{gather*}

If $d'<d+2$ (resp. $d'>d+2$), then for any $I\in \Lambda$, there exists some $j\in \{1,\cdots,d+2\}$ such that $i_j=0$ (resp. $i_j\ge 2$). 
Therefore, in both cases, we have $a_I=0$ for any $I\in \Lambda$, which implies that $F\cong T(S^{d}\circ\mathfrak{a}^{\#})\oplus T(S^{d'}\circ\mathfrak{a}^{\#})$ as $\A$-modules.
Hence, for $d'\neq d+2$, we have 
$$\Ext^1_{\AMod}(T(S^{d}\circ\mathfrak{a}^{\#}), T(S^{d'}\circ\mathfrak{a}^{\#}))=0.$$

If $d'=d+2$, then $a_I\neq 0$ requires $I=(1,\cdots,1)$, and thus we have 
\begin{gather*}
    F(\id_{d}\otimes\tilde{c})(x)= a x_1\cdots x_{d+2}
\end{gather*}
for some $a\in \K$.
For any morphism $f\in \A_1(m,n)$, there exists $g_0\in \A_0(m+2,n)$ such that $f=g_0\circ (\id_m\otimes \tilde{c})$ by Lemma \ref{decompositionofA}.
For any element $x'\in T(S^{d}\circ\mathfrak{a}^{\#})(m)$, there exists $h_0\in \A_0(d,m)$ such that $x'=T(S^{d}\circ\mathfrak{a}^{\#})(h_0)(x)$ since $x$ is a generator of $S^{d}\circ\mathfrak{a}^{\#}$. 
Therefore, by using \eqref{F=Ag0}, we obtain
\begin{gather*}
\begin{split}
    F(f)(x')&=F(g_0\circ (h_0\otimes \id_2))F(\id_d\otimes \tilde{c})(x)\\
    &=(S^{d+2}\circ\mathfrak{a}^{\#})(g_0\circ (h_0\otimes \id_2))(F(\id_d\otimes \tilde{c})(x)).
\end{split}
\end{gather*}
It follows that the functor $F$ is uniquely determined by $F(\id_{d}\otimes\tilde{c})(x)$, and we obtain 
$$\Ext^1_{\AMod}(T(S^{d}\circ\mathfrak{a}^{\#}), T(S^{d+2}\circ\mathfrak{a}^{\#}))\subset \K.$$

In order to prove that 
$\Ext^1_{\AMod}(T(S^{d}\circ\mathfrak{a}^{\#}), T(S^{d+2}\circ\mathfrak{a}^{\#}))\cong \K,$
we will construct a non-trivial extension $F$ of $T(S^{d}\circ\mathfrak{a}^{\#})$ by $T(S^{d+2}\circ\mathfrak{a}^{\#})$.
For $m\ge 0$, let $F(m)= S^{d}\circ\mathfrak{a}^{\#}(m)\oplus S^{d+2}\circ\mathfrak{a}^{\#}(m)$.
For a morphism $f_0\in \A_0(m,n)$, let $F(f_0)=S^{d}\circ\mathfrak{a}^{\#}(f_0)\oplus S^{d+2}\circ\mathfrak{a}^{\#}(f_0)$.
For a morphism $f\in \A_{\ge 2}(m,n)$, let $F(f)=0$.
For any morphism $f\in \A_1(m,n)$ and any elements $x'\in T(S^{d}\circ\mathfrak{a}^{\#})(m)$ and $y\in T(S^{d+2}\circ\mathfrak{a}^{\#})(m)$, 
let 
$$F(f)(x'+y)=F(f)(x')=(S^{d+2}\circ\mathfrak{a}^{\#})(g_0\circ (h_0\otimes \id_2))(x_1\cdots x_{d+2}),
$$
where the morphisms $g_0$ and $h_0$ are taken as in the last paragraph.
We will check that the map $F(f)$ for $f\in \A_1(m,n)$ is well defined.
For any $h_0,h_0'\in \A_0(d,m)$ such that $S^{d}\circ\mathfrak{a}^{\#}(h_0)(x)=S^{d}\circ\mathfrak{a}^{\#}(h_0')(x)$, 
we have 
\begin{gather*}
\begin{split}
    S^{d+2}\circ\mathfrak{a}^{\#}(h_0\otimes \id_2)(x_1\cdots x_{d+2})
    &=\mathfrak{a}^{\#}(h_0)(x_1)\cdots\mathfrak{a}^{\#}(h_0)(x_d)x_{d+1}x_{d+2}\\
    &=S^{d}\circ\mathfrak{a}^{\#}(h_0)(x)x_{d+1}x_{d+2}\\
    &=S^{d}\circ\mathfrak{a}^{\#}(h'_0)(x)x_{d+1}x_{d+2}\\
    &=S^{d+2}\circ\mathfrak{a}^{\#}(h_0'\otimes \id_2)(x_1\cdots x_{d+2}).
\end{split}
\end{gather*}
The assignment $F$ respects the relations \eqref{Casimir2-tensorDelta}, \eqref{Casimir2-tensorP} and \eqref{Casimir2-tensorAd} for Casimir $2$-tensors since we have
 \begin{gather*}
 \begin{split}
     F((\id_{d}\otimes P_{1,1})\circ (\id_d\otimes \tilde{c}))(x)
     &=S^{d+2}\circ\mathfrak{a}^{\#}(\id_{d}\otimes P_{1,1})(x_1\cdots x_{d+2})\\
     &=x_1 \cdots x_d x_{d+2}x_{d+1}\\
     &=F(\id_d\otimes \tilde{c})(x)
 \end{split}
 \end{gather*}
 and by the argument in the second paragraph of this proof, we have
 \begin{gather*}
    \begin{split}
      F((\id_{d}\otimes \Delta\otimes \id_1)\circ (\id_{d}\otimes \tilde{c}))(x)
    &=F((\id_{d+1}\otimes \eta\otimes \id_1+\id_{d}\otimes \eta\otimes \id_{2})\circ (\id_{d}\otimes \tilde{c}))(x)
    \end{split}
 \end{gather*}
and we have
\begin{gather*}
    \begin{split}
    F(\id_{d-1}\otimes ((ad\otimes ad)\circ(\id_1\otimes P_{1,1}\otimes \id_1)\circ (\Delta\otimes\tilde{c})))(x)
    =0=F(\id_{d-1}\otimes \tilde{c}\varepsilon)(x).
    \end{split}
\end{gather*}
Therefore, by the universality of the category $\A$ by Lemma \ref{presentationofA}, the assignment $F$ defines a $\K$-linear functor, which yields a non-trivial extension of $T(S^{d}\circ\mathfrak{a}^{\#})$ by $T(S^{d+2}\circ\mathfrak{a}^{\#})$.
This completes the proof. 
\end{proof}

\begin{remark}
A non-trivial extension of $T(S^{d}\circ\mathfrak{a}^{\#})$ by $T(S^{d+2}\circ\mathfrak{a}^{\#})$ can be realized as a sub-quotient of the $\A$-modules $\A(0,-)$ and $\A^{L}(L,H^{\otimes -})$, where $\A^{L}$ denotes the category of extended Jacobi diagrams in handlebodies, as follows.
For even $d=2d'$, a non-trivial extension can be realized as  
    \begin{gather*}
        \A_{\ge d'}(0,-)/(\A_{\ge d'+2}(0,-)+\A Q_{\ge d'}),
    \end{gather*}
where $\A_{\ge d}(0,-)$ is the $\A$-submodule of $\A(0,-)$ spanned by elements of degree $\ge d$, and where $\A Q$ is the $\A$-submodule of $\A_{\ge 2}(0,-)$ generated by the anti-symmetric element $Q_2=\tilde{c}^{\otimes 2}-P_{(23)}\circ \tilde{c}^{\otimes 2} \in \A_2(0,4)$ and $\A Q_{\ge d}=\A Q\cap \A_{\ge d}(0,-)$.
For odd $d=2d'+1$, a non-trivial extension can be realized as 
\begin{gather*}
      \A^{L}_{\ge d'}(L,H^{\otimes -})/(\A^{L}_{\ge d'+2}(L,H^{\otimes -})+\A^{L} Q_{\ge d'}),
\end{gather*}
where $\A^{L}_{\ge d'}(L,H^{\otimes -})$ is the $\A$-submodule of $\A^{L}(L,H^{\otimes -})$ spanned by elements of degree $\ge d$, and where $\A^{L} Q$ is the $\A$-submodule of $\A^{L}(L,H^{\otimes -})$ generated by 
$$i \otimes \tilde{c}- P_{(12)}\circ (i \otimes \tilde{c}),\quad i \otimes \tilde{c}^{\otimes 2}- P_{(34)}\circ (i \otimes \tilde{c}^{\otimes 2}).$$
See \cite[Section 6]{katadaAmod} for details.
\end{remark}

We compute the first Ext-groups between Schur functors and exterior power functors in $\AMod$.

\begin{theorem}\label{extLambda}
Let $\lambda$ be a partition and $d'\ge 0$.
Then the injective map
\begin{gather*}
     \iota:  \Ext^1_{\kgropMod}(S^{\lambda}\circ\mathfrak{a}^{\#}, \Lambda^{d'}\circ\mathfrak{a}^{\#})
     \hookrightarrow        
     \Ext^1_{\AMod}(T(S^{\lambda}\circ\mathfrak{a}^{\#}), T(\Lambda^{d'}\circ\mathfrak{a}^{\#}))
\end{gather*}
is an isomorphism.
Therefore, we have 
\begin{gather*}
    \Ext^1_{\AMod}(T(S^{\lambda}\circ\mathfrak{a}^{\#}), T(\Lambda^{d'}\circ\mathfrak{a}^{\#}))\cong \begin{cases}
        \K & \text{ if } \lambda=2^2 1^{d'-3}, 2 1^{d'-2}, 1^{d'+1},\\
        0 & \text{otherwise}.
    \end{cases}
\end{gather*}
\end{theorem}

\begin{proof}
Let $F$ be an extension of $T(S^{\lambda}\circ\mathfrak{a}^{\#})$ by $T(\Lambda^{d'}\circ\mathfrak{a}^{\#})$.
Set $l=l(\lambda)$ and let $x_{\lambda}\in S^{\lambda}\circ\mathfrak{a}^{\#}(l)$ be a generator of the simple $\kgrop$-module $S^{\lambda}\circ\mathfrak{a}^{\#}$.

Let $\{e_i\mid 1\le i\le l+2\}$ be a basis for $\Z^{l+2}$ and $\{x_i\mid 1\le i\le l+2\}$ the dual basis for $\mathfrak{a}^{\#}(l+2)=\Hom(\Z^{l+2},\K)\cong \K^{l+2}$.
Set
\begin{gather*}
    F(\id_{l}\otimes \tilde{c})(x_{\lambda})
    =\sum_{I=(i_1,\cdots,i_{l+2})\in \Lambda} a_I x_1^{i_1}\wedge \cdots\wedge x_{l+2}^{i_{l+2}}\in \Lambda^{d'}\circ\mathfrak{a}^{\#}(l+2),
\end{gather*}
where $a_I\in \K$ and $\Lambda=\{I=(i_1,\cdots,i_{l+2})\mid \sum_{j=1}^{l+2}i_j=d', i_1,\cdots,i_{l+2}\in \{0,1\}\}$.
In what follows, we will prove that $F(\id_l\otimes \tilde{c})(x_{\lambda})=0$.
If $d'>l+2$, then we have $\Lambda^{d'}\circ\mathfrak{a}^{\#}(l+2)=0$ and thus $F(\id_l\otimes \tilde{c})(x_{\lambda})=0$.
If $d'=l+2$, then we have 
$$F(\id_l\otimes \tilde{c})(x_{\lambda})=a_{(1,\cdots,1)} x_1\wedge \cdots\wedge x_{l+2}\in \Lambda^{l+2}\circ\mathfrak{a}^{\#}(l+2).$$
It follows from the relation \eqref{Casimir2-tensorP} that we have
\begin{gather*}
    \begin{split}
         F(\id_l\otimes \tilde{c})(x_{\lambda})
         &=F((\id_{l}\otimes P_{1,1})\circ (\id_l\otimes \tilde{c}))(x_{\lambda})\\
         &=F(\id_{l}\otimes P_{1,1})(a_{(1,\cdots,1)} x_1\wedge   \cdots\wedge x_{l+1}\wedge x_{l+2})\\
         &=\Lambda^{l+2}\circ\mathfrak{a}^{\#}(\id_l\otimes P_{1,1})(a_{(1,\cdots,1)} x_1\wedge \cdots\wedge x_{l+1}\wedge x_{l+2})\\
         &=a_{(1,\cdots,1)} x_1\wedge \cdots \wedge x_{l+2}\wedge x_{l+1}\\
         &=- a_{(1,\cdots,1)} x_1\wedge \cdots\wedge x_{l+1}\wedge x_{l+2}.
    \end{split}
\end{gather*}
Therefore, we have $a_{(1,\cdots,1)}=0$ and $F(\id_l\otimes \tilde{c})(x_{\lambda})=0$.
If $d'<l+2$, then for any $I\in \Lambda$, there exists some $j\in \{1,\cdots,l+2\}$ such that $i_j=0$.
For each $j\in \{1,\cdots,l+2\}$, we have 
\begin{gather*}
\begin{split}
    &F(\id_{j-1}\otimes \varepsilon \otimes \id_{l+2-j})F(\id_l\otimes \tilde{c})(x_{\lambda})\\
    &=F(\id_{j-1}\otimes \varepsilon \otimes \id_{d+2-j})(\sum_{I\in \Lambda} a_{I} x_1^{i_1}\wedge \cdots\wedge x_{l+2}^{i_{l+2}})\\
    &=\sum_{I\in \Lambda, \text{ s.t., } i_j=0} a_{I} x_1^{i_1}\wedge\cdots\wedge x_{j-1}^{i_{j-1}}\wedge x_{j}^{i_{j+1}} \wedge \cdots\wedge x_{l+1}^{i_{l+2}}.
\end{split}
\end{gather*}
On the other hand, we have 
$$F(\id_{j-1}\otimes \varepsilon \otimes \id_{l+2-j})F(\id_l\otimes \tilde{c})(x_{\lambda})=0$$
for $j=l+1,l+2$ by the relation \eqref{epsilonc} and for $j\in \{1,\cdots,l\}$ by 
\begin{gather*}
    \begin{split}
    &F(\id_{j-1}\otimes \varepsilon \otimes \id_{l+2-j})F(\id_l\otimes \tilde{c})(x_{\lambda})\\
    &=F(\id_{l-1}\otimes \tilde{c})F(\id_{j-1}\otimes \varepsilon \otimes \id_{l-j})(x_{\lambda})\\
    &=F(\id_{l-1}\otimes \tilde{c})((S^{\lambda}\circ\mathfrak{a}^{\#})(\id_{j-1}\otimes \varepsilon \otimes \id_{l-j})(x_{\lambda}))\quad (\text{by }\eqref{F=Ag0})\\
    &=0,
    \end{split}
\end{gather*}
where the last identity follows from the $\kgrop$-module structure of $S^{\lambda}\circ\mathfrak{a}^{\#}$.
Therefore, we have $a_I=0$ for any $I$ with $i_j=0$ for some $j$.
Hence, we have $F(\id_l\otimes \tilde{c})(x_{\lambda})=0$.

We will prove that for any morphism $f\in \A_1(m,n)$ and any element $x\in T(S^{\lambda}\circ\mathfrak{a}^{\#})(m)$, we have $F(f)(x)=0$.
Since $x_{\lambda}$ is a generator of $S^{\lambda}\circ\mathfrak{a}^{\#}$, there exists $h_0\in \A_0(l(\lambda),m)$ such that $x=T(S^{\lambda}\circ\mathfrak{a}^{\#})(h_0)(x_{\lambda})$.
There exists $g_0\in \A_0(m+2,n)$ such that $f=g_0\circ (\id_m\otimes \tilde{c})$ by Lemma \ref{decompositionofA}.
Then we have
\begin{gather*}
\begin{split}
    F(f)(x)&=F(f)(T(S^{\lambda}\circ\mathfrak{a}^{\#})(h_0)(x_{\lambda}))\\
    &=F(f\circ h_0)(x_{\lambda})\quad (\text{by } \eqref{F=Ag0})\\
    &=F(g_0 \circ (\id_{m}\otimes \tilde{c})\circ h_0)(x_{\lambda})\\
    &=F(g_0\circ (h_0\otimes \id_2)\circ (\id_{l(\lambda)}\otimes \tilde{c}))(x_{\lambda})\\
    &=F(g_0\circ (h_0\otimes \id_2))F(\id_{l(\lambda)}\otimes \tilde{c})(x_{\lambda})=0.
\end{split}
\end{gather*}
Therefore, by Lemma \ref{extensionAB}, we have $TU(F)\cong F$, which implies that the injective map $\iota$ is an isomorphism.
The last statement follows from Lemma \ref{extgrop}.
\end{proof}

\bibliographystyle{plain}
\bibliography{reference}

\end{document}